\documentclass{amsart}  
\usepackage{graphicx}
\usepackage {appendix}
\usepackage{amsfonts}
\usepackage{url}
\usepackage{hyperref} 
\hypersetup{backref,pdfpagemode=FullScreen,colorlinks=true}
\usepackage{amsmath}
\usepackage{amssymb}
\usepackage{amscd}
\usepackage{color}
\usepackage{amsthm}
\usepackage{bm}
\usepackage{indentfirst}
\usepackage{verbatim}  
\usepackage[hmargin=3cm,vmargin=3cm]{geometry}
\numberwithin{equation}{section}

\newtheorem{theorem}[equation]{Theorem}

\newtheorem{lemma}[equation]{Lemma}

\newtheorem{conjecture}{Conjecture}
\newtheorem{definition}[equation]{Definition}
\theoremstyle{definition}

\theoremstyle{remark}

\DeclareMathOperator{\image}{\mathrm{image}}

\begin{document}
% \href{http://yashamon.github.io/web2/papers/conformalsymplectic.pdf}{Direct link to author's version}
\title{A remark on deformation of Gromov non-squeezing}
\author{Yasha Savelyev}
% \thanks {Partially supported by PRODEP grant}
\email{yasha.savelyev@gmail.com}
\address{University of Colima, Faculty of science}
\begin{abstract} Let $R,r$ be as in the classical Gromov
non-squeezing theorem, and let	$\epsilon = (\pi R ^{2} - \pi
r ^{2})/ \pi r ^{2} $. We first conjecture that the Gromov
non-squeezing phenomenon persists for deformations of the
symplectic form on the range  $C ^{0}$ (w.r.t. the standard metric) $\epsilon $-nearby to the standard symplectic form.    
We prove this in some special cases, in particular when the
dimension is four and when $R <  \sqrt 2 r$. Given such a perturbation, we can no longer compactify the range and hence the classical Gromov argument breaks down. Our main method consists of a certain trap idea for holomorphic curves, analogous to traps in dynamical systems. 
\end{abstract}
\maketitle
 \section {Introduction}
One of the most important results to this day in symplectic
geometry is the so called Gromov non-squeezing theorem,
appearing in the seminal paper of Gromov~\cite{cite_GromovPseudoholomorphiccurvesinsymplecticmanifolds}. 
Let $\omega _{st} = \sum _{i=1} ^{n} dp _{i} \wedge dq _{i}$
denote the standard symplectic form on $\mathbb{R} ^{2n} $. 
Gromov's theorem then says that there does
not exist a symplectic embedding $$(B _{R}, \omega _{st})
\hookrightarrow (S ^{2} \times \mathbb{R} ^{2n-2}, \omega
_{\pi r ^{2}} \oplus \omega _{st}),
$$ for $R>r$, with
$ B _{R}  $ the standard closed radius $R$ ball in
$\mathbb{R} ^{2n} $ centered at $0$, and $\omega _{\pi
r ^{2}}$  a symplectic form on $S ^{2}$ with area $\pi r ^{2}$. 
It is very natural to conjecture the following simple
extension.
\begin{conjecture} \label{con_Gromov}
Let $R>r>0$ be given, set $\epsilon = (\pi R ^{2} - \pi r ^{2})/
\pi r ^{2}$ and let $\omega = \omega _{\pi r ^{2}} \oplus \omega _{st}$ be the
symplectic form on $M=S ^{2} \times \mathbb{R} ^{2n-2} $ as
above. Then for any symplectic form
$\omega'$ on $M$,  $C ^{0} $ $\epsilon $-close to $\omega$,
there is no symplectic embedding $\phi: (B _{R}, \omega
_{st}) \hookrightarrow (M, \omega')$, meaning that $\phi
 ^{*}\omega'=\omega _{st}$. \footnote {After the
 publication of this note, Spencer Cattalani has found
 a counterexample in dimensions higher than 4, that is $n>2$. We leave the conjecture here unmodified for consistency with the
 published version.}
\end{conjecture} 
Here, the $C ^{0}$ distance is with respect to the metric
$g _{J} = \omega (\cdot, J \cdot)$ on $M$ for $J$ the
standard complex structure, see \eqref{eq_dC0}. The above
$\epsilon $ is of course optimal, for if $c>
\epsilon $ there is a symplectic embedding of $B _{R}$ into
$(M, c \cdot \omega) $. It was pointed out to me by Spencer
Cattalani that the conjecture fails if we replace $S ^{2}$
by the radius $r$ disk, by a very simple argument appearing
in Gromov's original
~\cite{cite_GromovPseudoholomorphiccurvesinsymplecticmanifolds}. 

To prove this we cannot use the classical Gromov-Witten
argument since we cannot compactify the range.
Another idea is needed to get an appropriate compact moduli space of
holomorphic curves. One possible approach is to use
convexity or in other words the maximum principle. This approach can prove some
cases of the conjecture but is very unlikely to yield
a proof of the general case. This is because for a general $\omega'$
as above a compatible almost complex structure may be
forced to be non-standard at infinity. 

There is another approach that when sharpened should yield
the general case of the conjecture, at least in dimension 4. This is based 
on a simple idea of holomorphic traps
(Definition \ref{def:holomorphictrap}) somewhat analogous to
traps in dynamical systems. We use this to prove some special
cases of the conjecture.
\begin{theorem}  \label{thm:Gromov} 
When $n=2$ the conjecture above
holds in the following three cases:
\begin{enumerate}
\item $R< \sqrt 2 r$.
\item Let $p: S ^{2} \times \mathbb{R} ^{n} \to \mathbb{R}
^{2}  $ be the projection map. There is a continuous deformation
(topology as mentioned above) of symplectic forms $\{\omega _{t}\}$, $t \in
[0,1]$, $\omega' = \omega _{1}$, $\omega _{0} = \omega $. And such that the following is satisfied: there
is a compact $K \subset \mathbb{R} ^{2} $, such that for
each $x \in \mathbb{R} ^{2} - K$, and each $t$, $\omega _{t}$ is
non-degenerate on $p ^{-1} (x) $.
\item There is a continuous deformation of symplectic forms
$\{\omega _{t}\}$, $t \in
[0,1]$, $\omega' = \omega _{1}$, $\omega _{0} = \omega $ and
a continuous family of projections $\{p _{t}: M \to
\mathbb{R} ^{2} \}$ s.t. for each $t$ and $x \in \mathbb{R}
^{2} $ $\omega _{t}$ is non-degenerate on $p _{t} ^{-1}
(x)$.
\end{enumerate}
 
\end{theorem}

\section{A trap for holomorphic curves} \label{sec:A trap for holomorphic curves}
For basic notions of $J$-holomorphic curves we refer the
reader to \cite{cite_McDuffSalamonJholomorphiccurvesandsymplectictopology}. \begin{definition}\label{def:holomorphictrap}
Let $(M, J)$ be an almost complex manifold, and $A \in H _{2}
(M) $ fixed. Let $K \subset M$ be a closed subset. Suppose
that for every  $x \in \partial K$ (the topological
boundary) there is a ${J} $-holomorphic, real codimension
$2$, compact submanifold $H _{x} \ni x$ of $M$, satisfying:
\begin{itemize}
	\item  $H _{x} \subset K $.
	% \item $\pi (H _{x}) \subset M$  is a closed submanifold. 
	\item  $A \cdot H _{x} \leq 0$, where the left-hand side
	is the homological intersection number. 
\end{itemize}
We call such a $K$ a \textbf{\emph{$J$-holomorphic trap}}
(for class $A$ curves). 
\end{definition}
\begin{lemma} \label{lemma:K}
Let $M, J$ and $A$ be as above, and $K$ be a $J$-holomorphic
trap for class $A$ curves. Let $u: \Sigma \to M$ be
a $J$-holomorphic curve in class $A$, with $\Sigma
$ a connected closed Riemann surface. Then $$(\image u \cap
K) \neq \emptyset \implies \image u \subset K.$$  
\end{lemma}
\begin{proof}

Suppose that ${u} $ intersects $\partial K$, otherwise we
already have $\image {u} \subset interior (K) $, since $\image
{u}$ is connected (and by elementary topology).  Then ${u}
$ intersects $H _{x} $ as in the definition of a holomorphic
trap, for some $x$.
Consequently, as $A \cdot H _{x} \leq 0$, by positivity of
intersections \cite [Section
2.6]{cite_McDuffSalamonJholomorphiccurvesandsymplectictopology}, $\image u \subset H _{x} \subset  K$. 
\end{proof}

\section{Proof of Theorem \ref{thm:Gromov}}
\label{section:mainargument} 
\begin{definition} \label{def:compatible} A pair $(\omega, J)$ of a
2-form $\omega$ on a smooth manifold $M$ and an almost complex structure
$J$ on $M$ are \textbf{\emph{compatible}} if $\omega (\cdot,
J \cdot)$ defines a $J$-invariant inner product $g _{\omega,
J}$ on $M$. 
\end{definition}
Let us quickly recall the definition of the $C ^{0}$  distance 
$d _{C ^{0}}$, on the set of 2-forms $\Omega ^{2}
(M) $ for a fixed metric $g$ on $M$.
\begin{equation} \label{eq_dC0}
d _{C ^{0}} (\omega _{0}, \omega _{1}) = \sup _{|v \wedge w| _{g} = 1} |\omega
_{0} (v, w) - \omega _{1} (v,w)   |,
\end{equation}
where more specifically, the supremum is over all
$g$-norm 1 simple bivectors $v \wedge w$ in $\Lambda ^{2} (TM) $.  

Let $\omega$ be the symplectic form on $M = S ^{2} \times
\mathbb{R} ^{2}  $ as in the statement of Conjecture
\ref{con_Gromov}. In our case $d _{C ^{0}}$  will be defined with respect to the metric $g _{\omega,J}$  as in Definition \ref{def:compatible} for $J$ the standard product complex structure.

We now prove the second case of the theorem.
Let $\epsilon =(\pi R ^{2} - \pi r ^{2})/ \pi r ^{2}$.
Suppose by contradiction that there is a $d _{C
^{0}}$-continuous family $\{\omega _{t}\}$ of symplectic
forms s.t. 
\begin{itemize}
\item $d _{C ^{0}} (\omega, \omega _{1}) < \epsilon$.
\item There exists a symplectic embedding $$\phi: (B _{R},
w _{st}) \hookrightarrow (M, \omega _{1}). $$ 
\item For each $t \in [0,1] $, $\omega _{t}$ is
	non-degenerate on the fibers $M _{x}$ of
	the projection $$p: M  \to \mathbb{R} ^{2}, $$ for $x \in
	\mathbb{R} ^{2} - K'$ for some
	compact $K' \subset \mathbb{R} ^{2} $.  
\end{itemize}
Set $B:= \phi (B _{R}) $ and let $D ^{\circ} \supset (p(B)
\cup K')$ be an open standard disk in $\mathbb{R} ^{2} $, and
let $D$ denote its closure.
So $K = S ^{2} \times D$ is a compact subset of $M$, with
the properties:
\begin{enumerate}
	\item $\partial K$ is smoothly foliated by the fibers $M
	_{x}$.
	\item For each $t$, $\omega  _{t}$ is non-degenerate on the fibers $M _{x}$ contained in $\partial K$.
\end{enumerate}
We denote by $T ^{vert} \partial K \subset TM  $, the sub-bundle of vectors tangent to the leaves of the above-mentioned foliation. 
%
% We will first setup certain kinds of domains, in
% the sense of Theorem \ref{thm:noSkycatastrophe}.  

Let $j$ be the standard complex structure on $B _{R}$. We may extend $\phi _{*}j $ to an $\omega _{1} $-compatible
almost complex structure $J _{1} $ on $M$, preserving $T ^{vert} \partial K $ using:
\begin{itemize}
  \item $\image \phi$ does not
intersect $\partial K$.
	\item The non-degeneracy of $\omega _{1}$ on the fibers. 	\item The well known existence/flexibility results for
	compatible almost
complex structures on symplectic vector bundles, see for
instance
~\cite[Section
2.6]{cite_McDuffSalamonIntroductiontosymplectictopology}. 
 \end{itemize}
We may then extend $J _{1}$ to an appropriately
\footnote{Because $M$ is not compact, we need to treat the
space of almost complex structures as a nuclear LF manifold,
~\cite{cite_MichorNuclearLF} rather than a Frechet manifold.
However, this is only cosmetic (in the setup of the moment) since we are only interested
in the behavior over a fixed compact set. Thus, we could
also use a modified Frechet topology induced by choosing
a fixed compact.} smooth family $\{J _{t} \}  $,
$t \in [0,1] $, of almost complex structures on $M$, s.t. $J
_{t} $ is $\omega _{t} $-compatible for each $t$, with $J
_{0} = J $ as above, and such that $J
_{t} $ preserves $T ^{vert} \partial K $ for each $t$. 
The latter condition can be satisfied by similar reasoning
as above, using that $\omega _{t}$ is non-degenerate on the
fibers $M _{x}$, contained in $\partial K$ for each $t$.  

Such fibers are $J _{t}$-holomorphic hypersurfaces for
each $t$, and smoothly foliate $\partial K$.
Moreover, if $A = [S ^{2} ] \otimes [pt] \in H _{2} (M)$  then the intersection number of $A$ with a fiber
is 0. That is $A \cdot p ^{-1} (z) = 0$, for $\forall z \in
\mathbb{R} ^{2} $.   And so $K$ is a compact $J _{t}$-holomorphic trap for class $A$ curves, for each $t$.  

Set $x _{0} := \phi (0) $. Denote by $\mathcal{M} _{t}$ the
space of equivalence classes of maps $u: \mathbb{CP} ^{1}
\to M$, where $u$ is a $J _{t}$-holomorphic, class $A$ curve passing through $x _{0}$.
The equivalence relation is by the usual
biholomorphism reparametrization group action, so that $u
\sim u'$ if there exists a biholomorphism $f:
\mathbb{CP} ^{1} \to \mathbb{CP} ^{1} $ s.t. $u' = u \circ
f$.  Then $\mathcal{M} = \cup _{t} \mathcal{M}
_{t}$  is compact by energy minimality of $A$ (which rules
out bubbling), by Lemma \ref{lemma:K}, and by compactness of
$K$. 

To use Gromov-Witten type curve counts, we need to regularize. We may use the ``standard'' Banach
approach. This has the advantage of being readily understood
by experts but a possible disadvantage of appearing opaque
and ad hoc to new-comers to the field.  
For this reason
we will also give an independent argument using polyfold theory.
\subsection{Banach approach} \label{sec_Banach approach}
This is based on ~\cite{cite_McDuffSalamonJholomorphiccurvesandsymplectictopology}
and the picture is as follows.
Let $\mathcal{B} $ be the universal Banach moduli space of
class $A$ curves:
\begin{equation*}
\mathcal{B} = \mathcal{M} ^{*} (A, \mathcal{J} ^{l})
: = \{(u, J) \,|\, J \in \mathcal{J} ^{l},  \text{$u:
\mathbb{CP} ^{1} \to M$ is a simple class $A$
$J$-holomorphic curve} \},
\end{equation*}
where $\mathcal{J} ^{l}$ is the space of class $C ^{l}$
almost complex structures, taking $l$ to be sufficiently large.
Then we have an evaluation map $ev: \mathcal{B} \to M$, $(u,
J) \mapsto u (z _{0})$, for $z _{0} \in \mathbb{CP} ^{1}
$ fixed.   
Let $\pi _{}: \mathcal{B} \to \mathcal{J} ^{l}$ be the
projection.

The product map $$\mathcal{B} \xrightarrow{ev \times
\pi _{} } M \times \mathcal{J} ^{l}$$ is a Fredholm map.
There is one immediate problem: given $(x _{0}, J) \in M \times
\mathcal{J} ^{l}$ a priori we may not be able to perturb it to
a regular value of the form $(x', J')$ (that is we may need
to perturb $x _{0}$ to $x' $). 
% TODO: {mention this}
This would complicate the last step of the
proof of the theorem, which needs specifically a holomorphic
curve through $x _{0}$. 
% Note that abstract perturbation
% techniques as in ~\cite{cite_Pardon} do not instantaneously
% solve the problem, as we need to construct 
% adapted implicit structures specific to our problem (we have
% a compactness issue).
Fortuitously, it turns out that the map $ev$ is always
a submersion, see
[Proposition
3.4.2]~\cite{cite_McDuffSalamonJholomorphiccurvesandsymplectictopology}. 
Thus, there is no need to perturb $x _{0}$.
\begin{lemma} \label{lemma_}
Let $\{J _{t}\}$, $t \in [0,1]$, be the family as
constructed above.  Then there is a path $p': [0,1] \to
M \times \mathcal{J} ^{l}$, $t \mapsto (x _{0}, J'_{t})$,
such that:
\begin{enumerate}
	\item $ev \times \pi _{} $ is transverse to $p'$ in the standard
	differential topology sense, (this is equivalent to $\{J' _{t}\}$
	being a regular homotopy, as defined in
	~[Definition
	3.1.7]\cite{cite_McDuffSalamonJholomorphiccurvesandsymplectictopology}).  
	\item $J' _{t} $ is $\omega _{t}$-compatible for
	each $t$.
	\item $J' _{t} $ preserves $T ^{vert} \partial K $ for
	each $t$. \label{cond_preserveses}
	\item $J'  _{0} = J$.
\end{enumerate}
\end{lemma}
\begin{proof} [Proof]
Only condition \eqref{cond_preserveses} requires an explanation. To see that this can be satisfied,
take open $U _{1}, U _{2} \subset M$, homeomorphic to an open ball in
$\mathbb{R} ^{4} $, with $\overline{U}  _{1} \subset U _{2}$
with $B \subset U _{1}$ and with $U _{2} \subset K$.
Now any closed $J$-holomorphic
curve which intersects $U _{1}$ must intersect $U _{2}
- \overline{U}  _{1}$. As otherwise we contradict $H _{2} (U
_{1}, \mathbb{Z}) =0 $, since the homology class of closed,
non-constant, $J$-holomorphic curve is never zero, for $J$ compatible with a symplectic form.
Thus, the family $\{J _{t}\}$ can be 
be regularized by perturbing only within the region $U _{2}
- \overline{U}  _{1}$ cf. 
~\cite[proof of
Lemma 3.4.4]{cite_McDuffSalamonJholomorphiccurvesandsymplectictopology}. In this case, \eqref{cond_preserveses} will
be automatically satisfied. 
\end{proof}
For $p'$ as in the lemma, define $\mathcal{M}'$ to be the preimage $(ev \times \pi _{}) ^{-1} (\image p')/ \sim$,  where
$\sim $ is the following equivalence relation. $u \sim u'$
if there is a biholomorphism $f:
\mathbb{CP} ^{1} \to \mathbb{CP} ^{1} $ s.t. $u' = u \circ
f$ and s.t. $f (z _{0}) = z _{0}$. Then $\mathcal{M}'$ is
a compact one dimensional manifold. 

The boundary component $(ev \times \pi _{}) ^{-1} (x _{0}, J)/ \sim$ is
a point corresponding to the single $J$-holomorphic, class
$A$ curve passing through $x _{0}$. It follows that the
boundary component $(ev \times \pi _{}) ^{-1} (x _{0}, J')/ \sim$
is likewise non-empty. Then let $u _{0} \in (ev \times \pi
_{} ) ^{-1} (x _{0}, J')/ \sim$, we will use this
further ahead.

\subsection{Polyfold approach} \label{sec_Polyfold approach}
% Note that in the classical Gromov setting, where $M$ is
% compact we can instead regularize the unconstrained moduli
% space, and then take the homological intersection of the
% evaluation at a point (i.e. take the Gromov-Witten
% invariant). But we cannot do that here, since the
% unconstrained moduli space in our case is not compact.
Alternatively, we may use Hofer-Wysocki-Zehnder polyfold
regularization in Gromov-Witten theory, especially as
recently worked out in this present context by the team of
Franziska Beckschulte, Ipsita Datta, Irene Seifert,
Anna-Maria Vocke, and Katrin Wehrheim. We can also of course
use other virtual approaches, but this is not instantaneous,
for example if we were to invoke Pardon~\cite{cite_Pardon} then we
would have needed to construct implicit
atlases in the constrained case (this can be done of
course).

As explained in ~\cite[Section 3.5]{cite_PolyfoldGromov}, in a essentially identical situation, we may embed
$\mathcal{M}$ into a natural polyfold setup of
Hofer-Wysocki-Zehnder ~\cite{cite_HWZapplicationsPolyfolds}.
More to the point, we express $\mathcal{M}$ as the zero set of an
sc-Fredholm section of a suitable (tame, strong) $M$-polyfold bundle.  
The only difference with the setup of ~\cite[Section 3.5]{cite_PolyfoldGromov} is that they compactify 
$M$ to $S ^{2} \times T ^{2}$, to get a compact moduli
space. We of course cannot
compactify, but remember that we used the
holomorphic trap idea to force compactness of
$\mathcal{M}$. And so we are in an equivalent situation.

Again as in ~\cite{cite_PolyfoldGromov}, we take the $M$-polyfold regularization of $\mathcal{M}$. This gives a one
dimensional compact cobordism $\mathcal{M} ^{reg}$  between $\mathcal{M} ^{reg} _{0}$ and
$\mathcal{M} ^{reg} _{1}$.  

Now $\mathcal{M} ^{reg} _{0}$ is a point: corresponding to the unique $(J=J _{0})$-holomorphic class $A$,  curve $u: \mathbb{CP} ^{1} \to M$ passing through $x
_{0}$.  Consequently, $\mathcal{M} ^{reg} _{1}$ is non-empty, that is there is a $J _{1}$-holomorphic class $A$ curve $u _{0}:
\mathbb{CP} ^{1} \to M$ passing through $x _{0}$.
\subsection{Finishing the proof} \label{sec_Finishing the proof}
% \begin{remark} \label{remark:polyfold}
% 	It is certainly possible that more classical, geometric
% 	perturbation style arguments may be adopted to
% 	the present problem. There are however difficulties: it is important for us to work with curves
% 	constrained to pass through a specific point, instead of doing
% 	homological intersection of an unconstrained evaluation
% 	cycle, with a point (as in the classical proof of
% 	Gromov non-squeezing). For without the specific constraint our moduli
% 	space is not even compact, and hence the
% 	homological intersection theory makes no sense. Such
% 	a constraint may not neatly fit into classical analytical
% 	framework of McDuff-Salamon
% 	~\cite{cite_McDuffSalamonJholomorphiccurvesandsymplectictopology}. 
% \end{remark}

Now $\langle \omega, A \rangle = \pi _{} \cdot r ^{2}
$, and we have a representative for $A$ whose $g$-area is
$\pi r ^{2}$.
So we have: $$  |\langle {\omega} _{1}, A  \rangle - \pi \cdot r ^{2}| = |\langle {\omega} _{1}, A  \rangle - 
\langle {\omega}, A  \rangle | < \epsilon  \pi \cdot r ^{2}
=\pi R ^{2} - \pi r^{2},$$
where the inequality uses that $d _{C ^{0}} ({\omega}, {\omega} _{1} ) <\epsilon$. 
So we get $$|\int _{\mathbb{CP} ^{1} } u _{0} ^{*} \omega
_{1} - \pi r ^{2}    |   < \pi R ^{2} - \pi r^{2}.   $$   
And consequently, $$\int _{\mathbb{CP} ^{1} } u _{0} ^{*}
\omega _{1} < \pi R ^{2}. $$

We may then proceed exactly as in the now classical proof
of Gromov~\cite{cite_GromovPseudoholomorphiccurvesinsymplecticmanifolds.} of the non-squeezing theorem to get a contradiction and
finish the proof. More specifically, $\phi ^{-1}
({\image \phi \cap \image u _{0}})  $ is a  minimal surface in
$B _{R}  $, with boundary on the boundary of $B _{R} $,
and passing through $0 \in B _{R} $. By construction it
has area strictly less than $\pi R ^{2} $, which is
impossible by the classical monotonicity theorem of
differential geometry. See also
\cite[Lemma A.2]{cite_PolyfoldGromov} where the
monotonicity theorem is suitably generalized, to better
fit the present context.
	
This finishes the proof of the second case. 
To prove the first case, note that if $v$ is a $g$-unit
vector then $\omega (v, Jv) = 1$. If $\epsilon = (\pi R ^{2}
- \pi r^{2})/ \pi _{} r ^{2} $ and $\pi _{} R ^{2} < 2 \pi
r ^{2} $ then $\epsilon <1$.
And so if $\omega'$ is $\epsilon$ close to $\omega $ then
$\omega'  (v, Jv) >0$. It follows that:
\begin{itemize}
	\item $\omega _{t} = (1-t) \omega + t \omega'$ is non-degenerate, for each $t \in [0,1] $.  
	\item For each $t \in [0,1] $, $\omega _{t}$ is
	non-degenerate on the fibers $M _{x}$ of
	the natural projection $$p: (M = S ^{2} \times \mathbb{R}
	^{2}) \to \mathbb{R} ^{2}, $$  for all $x \in \mathbb{R}
	^{2} $.
\end{itemize}
So the family $\{\omega _{t}\}$ satisfies the hypothesis of
the second case taking $K= \emptyset  $, and the conclusion follows. 

For case 3, suppose we are given such a family of projections $p _{t}$,
and let $D \subset \mathbb{R} ^{2} $, be a closed disk
constructed as in the proof of case 2. Define $\widetilde{p}: M \times [0,1]
\to \mathbb{R} ^{2} $ by $\widetilde{p} (x,t) = p _{t} (x)$,
then by assumptions $\widetilde{p} $ is continuous. 
Define the compact subset of $M \times [0,1]$:
$$K _{D} = \widetilde{p} ^{-1} (D). $$ 
As in the proof of case 2, we get  
a family $\{J _{t}\}$ of $\omega _{t}$ compatible
almost complex structures, s.t. $K _{D} $ is trapping for corresponding cobordism
moduli space. That is for each $(u,t) \in \mathcal{M} $ the
image of $u$ is contained in $K _{D} \cap M \times \{t\}$.
It follows that $\mathcal{M} $ is compact, then proceed
as in the proof of case 2.
\qed
\section{Acknowledgements} 
% TODO: {fix}
I am grateful to Helmut Hofer, Bulent Tosun,  and Semon
Rezchikov for an interesting discussion as well as Felix
Schlenk, Misha Gromov and Dusa McDuff
for some feedback. Thanks also to the referee for careful
proof reading. 
% I give thanks to Kevin Sackel for a number of discussions on
% related topics.
% interest and extensive discussions of related topics, to Emmy Murphy and Richard Hind for
% early feedback, and Vestislav Apostolov for kindly answering
% some questions.  
% I thank Yong-Geun Oh for a number of discussions on related topics, and for an invitation to IBS-CGP, Korea. Thanks also to John Pardon for receiving me during a visit in Princeton. I also thank Dusa McDuff for comments on earlier versions and Kaoru Ono, Emmy Murphy, Viktor Ginzburg, Yael Karshon, Helmut Hofer and Richard Hind, for helpful discussions on related topics. 
\bibliographystyle{siam}  
%  \bibliography{/root/texmf/bibtex/bib/link}  
 % \bibliography{/home/yashasavelyev/texmf/bibtex/bib/link} 
\bibliography{link} 
% \bibliography{/home/yasha/texmf/bibtex/bib/link} 
\end{document}